\documentclass[11pt]{article}

\usepackage{setspace}
\usepackage{tikz}
\usetikzlibrary{arrows}
\usepackage{framed}
\usepackage{cite}
\usepackage{amsthm,amsfonts,epsfig,graphics,amsbsy,subfigure,soul,mathtools,bold-extra}
\usepackage[color=green!40]{todonotes}

\usepackage[left=3cm,top=3cm,right=3cm,bottom=3.4cm]{geometry}

\usetikzlibrary{arrows,shapes,positioning}
\usetikzlibrary{decorations.markings}
\tikzstyle arrowstyle=[scale=1]
\tikzstyle directed=[postaction={decorate,decoration={markings,
    mark=at position .65 with {\arrow[arrowstyle]{stealth}}}}]
\usepackage{color}		
\usepackage{epsfig}
\usepackage{soul}
\usepackage{floatrow}
\usepackage{enumitem}
\usepackage{pifont}
\usepackage{mathtools}
\newcommand{\intL}{\int\limits }
\newcommand{\half}{^\infty_0 }
\newcommand{\RR}{\mathbb{R}}
\newcommand{\NN}{\mathbb{N}}
\newcommand{\uu}{\mathbf{u}}
\newcommand{\xx}{\mathbf{x}}
\newcommand{\yy}{\mathbf{y}}

\newcommand{\ww}{\mathbf{w}}

\newcommand{\bbeta}{\boldsymbol{\beta}}
\newcommand{\aalpha}{\boldsymbol{\alpha}}

\newcommand{\XXi}{\boldsymbol{\Xi}}
 \newcommand{\rmd}[1]{\mathrm d#1}
 
  \newcommand{\Co}{\mathbf  C}
  \newcommand{\Po}{\mathbf  P}
  \newcommand{\sph}{\mathbb S}
  \newcommand{\cone}{\operatorname{conv}}

\usepackage{amssymb}
\usepackage{amsmath}
\usepackage{graphicx}
\newtheorem{thm}{Theorem}
\newtheorem{defi}[thm]{Definition}
\newtheorem{rmk}[thm]{Remark}

\newtheorem{lem}[thm]{Lemma}

\usepackage{authblk}

\title{The conical Radon transform with  vertices on triple lines}
 \date{}
 
\author{Sunghwan Moon\thanks{corresponding author}}

\affil{Department of Mathematics, Kyungpook National University, \authorcr
Daegu 41566, Republic of Korea ({\tt sunghwan.moon@knu.ac.kr})}

 \author{Markus~Haltmeier}

\affil{Department of Mathematics, University of Innsbruck,
Technikerstrasse 13, \authorcr6020 Innsbruck, Austria ({\tt markus.haltmeier@uibk.ac.at})}

\begin{document}
\maketitle 
\begin{abstract}
We study the inversion of  the conical Radon which integrates a function in three-dimensional  space from integrals over circular cones.  The conical Radon recently got significant attention due to its relevance in various imaging applications such as Compton camera imaging and single scattering optical tomography.
The unrestricted conical Radon transform is over-determined because the manifold of all cones depends on six variables: the center position, the axis orientation and the  opening angle of the cone. 
In this work, we consider a particular  restricted transform using triple line sensor where integrals over a three-dimensional set of cones  are collected,  determined by a one-dimensional vertex set, a one-dimensional set of  central axes, and the one-dimensional set of opening angle.  As the main result in this paper we 
derive an analytic inversion formula for the restricted conical Radon transform. 
Along that way we  define a certain ray transform adapted to the triple line sensor 
for which we  establish an analytic inversion formula.    
\end{abstract}

\smallskip\noindent{\bf Keywords.}
Conical Radon transform, Compton camera, inversion, reconstruction formula.   

\smallskip\noindent{\bf AMS classification numbers.}
44A12; 65R10; 92C55

\section{Introduction}

The conical Radon transform maps a function 
$f \colon \RR^3 \to \RR$ in three-dimensional space  to its integrals over one-sided  circular cones,
\begin{equation*}
 \Co f (\uu, \bbeta,\psi) =  \int\limits_{\sph^2}\int\limits^\infty_0 f(\uu + r\aalpha )r \delta(\aalpha
\cdot \bbeta - \cos\psi) \rmd r {\rm d}S(\aalpha) 
\end{equation*}
for $(\uu, \bbeta,\psi) \in \RR^3 \times \sph^2 \times [0, \pi ]$.
Here the cones of which the function is integrated are described by the vertex  
$\uu \in \RR^3$, the central axis $\bbeta \in \sph^2$  and the opening angle 
$\psi \in [0, \pi]$, and $\delta$ denotes the one-dimensional delta-distribution.  
Inverting the unrestricted  conical Radon transform is over-determined as 
$\Co f $  depends on six variables whereas the unknown function only depends on three.
Various forms of conical Radon transforms arise by restricting to certain subsets of cones. 
Several inversion formulas for various types of conical transforms have derived in~\cite{baskozg98,cebeiromn15,creeb94,gouia14,gouiaa14,haltmeier14,maxim2013filtered,maximfp09,moonip15,moonsiims16,moonsima16,nguyentbd04,nguyentg05,schiefenederh17,smith05,terzioglu15,truongnz07}.
Also, as a special two-dimensional version of the conical Radon transform, the $V$-line transform is also studied in \cite{moonieee17,mooncma13}.
For a recent  review  of conical Radon transforms see \cite{ambartsoumian20197,terzioglukk18,nguyent18}.
In this paper we  restrict the cones of integration  to  a three dimensional submanifold  of conical surfaces associated to  linear detector where vertices and the axes directions are restricted to one dimension.

\begin{figure}[htb!]
\begin{center}
  \begin{tikzpicture}[>=stealth]
  

    \draw[black, thick] (-3,-0.2)--(7,-0.2)--(10,0.6)--(7+3/8,0.6);
   \draw[black, thick] (0-3/8,0.5)--(-3,-0.2)  ;
   \draw[black,dashed,thick] (0-3/8,0.5)--(0,0.6)--(7+3/8,0.6);
      \draw[black, thick] (-3,0.5) -- (7,0.5)--(10,1.3)--(0,1.3)--(-3,0.5) ;
   
   \draw[blue,directed,thick] (3.5+1.5/1.4141,3.5-1.5/1.4141) -- (1,1) ;
   \draw[blue,thick] (1,1) -- (3.5-1.5/1.4141,3.5+1.5/1.4141) ;
   \draw[thick,->,red] (0.1,0.1) -- (3.5,3.5) ;
      
      \fill[black]  (1,1) circle (0.08);
      \fill[black]  (0.1,0.1) circle (0.08);
      
   \draw[blue,rotate around={-45:(3.5,3.5)},dashed,thick] (5,3.5) arc (0:360:1.5cm and 0.5cm);

   \draw (2.1,2.1) arc (30:80:21pt);
   \node at (5,1) {scattering plane};
   \node at (5,0.1) {absorption plane};
   \node at (0.7,1.1) {$\uu$};
   \node at (-0.2,0.2) {$\uu_a$};
   \node at (3.7,3.4) {$\bbeta$};
   \node at (1.7,2) {$\psi$};
   
     \fill[green!50!black]  (3.5+1.5/1.4141,3.5-1.5/1.4141) circle (0.08);

  \end{tikzpicture}
  \caption{Schematic representation of a standard Compton camera}\label{fig:compton}
\end{center}
\end{figure}
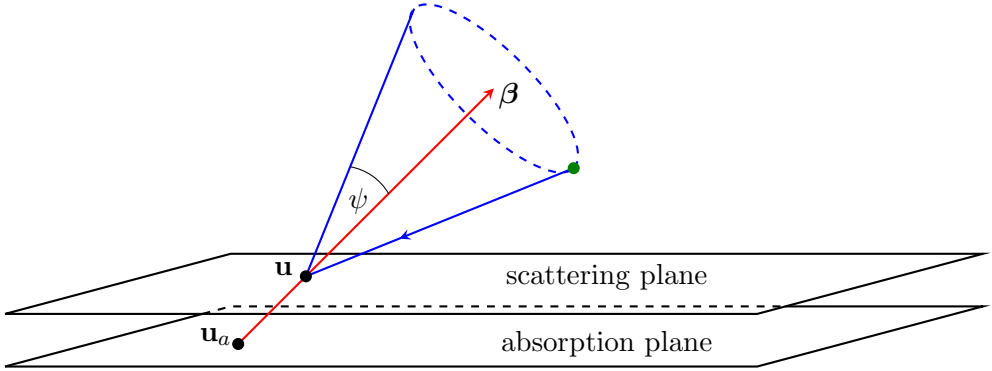

Among others, inverting the conical Radon transform   is relevant for Compton camera imaging.  
A Compton camera (also called electronically collimated $\gamma$-camera) has been  proposed  in \cite{singh83,toddne74} for single photon emission computed tomography (SPECT)  offering  increased efficiency compared  to a conventional $\gamma$-camera.
A standard Compton camera consists of two planar detectors: a scatter detector and an absorption detector, positioned one behind the other.  A photon emitted from a radioactive source toward the camera undergoes Compton scattering in the scatter detector, and is absorbed in the absorption detector positioned behind (see Figure~\ref{fig:compton}). 
In each detector plane, the positions $\uu$, $\uu_{\rm a}$ and the energy of the photon are measured.
The energy difference determines the scattering angle $\psi$ under which the photon path has been scattered at the scattering detector. Therefore, the measurements allow  to conclude that photon  has been emitted  on a conical surface with vertex $\uu$, axis direction pointing form $\uu_{\rm a}$ to $\uu$ and opening angle $\psi$. In a similar manner, assuming a continuous  source distribution of emitting photons, the  Compton camera yields the conical Radon transform  of source distribution with vertices restricted to the scattering plane. The corresponding data  depend on five variables.

Instead of planar Compton cameras in this paper we consider linear Compton cameras consisting of two parallel detector lines (left image in Figure~\ref{fig:comptonwithline}). 
Basically, the data acquisition with linear detectors is the same as in the standard one. The only difference is that the vertices  are restricted to the one-dimensional scattering  detector and the axes directions are restricted to the one-dimensional set of all directions  pointing from the linear absorption detector to the linear scattering detector.  Thus, the corresponding data depend on three  variables and thus are no longer over-determined.  
As the main theoretical result of this paper, we derive an analytic inversion formula for triple linear sensor.
As shown in the right image in Figure~\ref{fig:comptonwithline}, the  triple line sensor consists  three one dimensional vertex sets $\XXi_1$,  $\XXi_2$,  $\XXi_3$ each associated with  a one-dimensional set  axis directions.

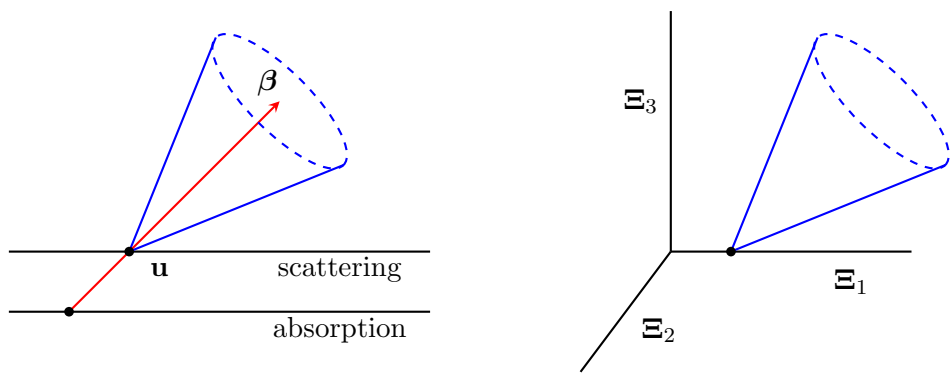
\begin{figure}[htb!]
\begin{center}
  \begin{tikzpicture}[>=stealth,scale=0.8]
  

   \draw[thick] (-1,0) -- (6,0) ;
      \draw[thick] (-1,1) -- (6,1) ;
   \draw[blue,thick] (3.5+1.5/1.4141,3.5-1.5/1.4141) -- (1,1) ;
   \draw[blue,thick] (1,1) -- (3.5-1.5/1.4141,3.5+1.5/1.4141) ;
   \draw[red,thick,->] (0,0) -- (3.5,3.5) ;
   
         \fill[black]  (1,1) circle (0.08);
         \fill[black]  (0,0) circle (0.08);

   \draw[blue,rotate around={-45:(3.5,3.5)},dashed,thick] (5,3.5) arc (0:360:1.5cm and 0.5cm);
   \node at (4.5,0.7) {scattering};
   \node at (4.5,-0.3) {absorption};
   \node at (1.5,0.7) {$\uu$};
   \node at (3.3,3.8) {$\bbeta$};
   
  \draw[thick] (10,1) -- (14,1) ;  \node at (13,0.5) {$\XXi_1$};  
  \draw[thick] (10,1) -- (8.5,-1) ; \node at (9.8,-0.3) {$\XXi_2$}; 
  \draw[thick]  (10,1) -- (10,5) ; \node at (9.5, 3.5) {$\XXi_3$};

   \draw[blue,thick] (10+3.5+1.5/1.4141,3.5-1.5/1.4141) -- (11,1) ;
   \draw[blue,thick] (11,1) -- (10+3.5-1.5/1.4141,3.5+1.5/1.4141) ;
     \fill[black]  (11,1) circle (0.08);
   \draw[blue,rotate around={-45:(13.5,3.5)},dashed,thick] (15,3.5) arc (0:360:1.5cm and 0.5cm);   

  \end{tikzpicture}
  \caption{Left: Schematic representation of a Compton camera with line detectors. 
  Right: Triple line detector consisting of three orthogonal lines.}\label{fig:comptonwithline}
\end{center}
\end{figure}

In practice it may  be easier to build linear detectors than planar detectors because the former requires less physical space and less complicated electronics.
Moreover, to the high dimensionality  of the data obtained from the planar detectors methods only using  partial data have been derived (see e.g. \cite{baskozg98, cebeiromn15, creeb94, haltmeier14, moonsima16, nguyentg05, smith05}).
On the other hand, Compton camera data are considerably noisy and utilizing full 
five-dimensional data is advised to obtain accurate reconstruction results  \cite{allmarasdhkk13,feng2019}. 
However, the data of planar sensors can be grouped in data sets of several virtual linear detectors. 
Therefore, inversion methods for line  detectors can  be applied to give several reconstructions 
for planar detectors that can be aggregated for noise reduction. 

The rest of this paper is organized as follows. The conical Radon transform  with triple 
linear detector in introduced in Section~\ref{sec:formulation}. 
In Section~\ref{sec:inversion}  we derive an analytic inversion formula.  As main ingredient of 
 the proof we reduce the conical Radon transform to a  weighted  ray transform  and proof a novel inversion formula for this ray transform.  The paper concludes with a short  summary 
and outlook presented in Section~\ref{sec:conclusion}.   

\section{The conical Radon transform}
\label{sec:formulation}

In this section, we formally define the  conical Radon transform with vertices on triple lines.
Let 
\begin{equation}\label{eq:xi}
	\XXi = \XXi_1 \cup \XXi_2 \cup \XXi_3
\end{equation} 
be the set of vertices where
\begin{equation} \label{eq:xi1}
\begin{aligned}
\XXi_1 &\coloneqq  \{(y_1,0,0) \colon y_1\in[0,1]\}  \\ 
\XXi_2 &\coloneqq   \{(0,y_2,0) \colon y_2\in[0,1]\} \\ 
\XXi_3 &\coloneqq   \{(0,0,y_3) \colon y_3\in[0,1]\} \,.
\end{aligned}
\end{equation}
We denote by  $\psi\in [0,\pi]$  the opening angle of the circular cone and consider for each of the 
sets $\XXi_j$ a different one-dimensional set of central axes.       
More precisely we parametrize each of the sets of central axes with  
$\bar \bbeta = (\beta_1,\beta_2) \in \sph^1$ and define  
\begin{equation} \label{eq:beta}
\bbeta_{\uu}=
\begin{cases}
(\beta_1,0,\beta_2)  & \text{ if   $\uu\in \XXi_1$} \\
(0,\beta_1,\beta_2)  & \text{ if   $\uu\in \XXi_2$} \\
(\beta_1,0,\beta_2)  & \text{ if   $\uu\in \XXi_3$}  \,.
\end{cases}
\end{equation}
Let $f \colon \RR^3  \to \RR$ be the distribution of the radioactivity sources.

For the following it is  convenient to work with  $s = \cos\psi\in[-1,1]$. We then 
define the conical Radon transform $\Co_k f$ of a function $f\in C(\RR^3)$ with compact support 
as follows.

\begin{defi}[Conical Radon transform]
For given $k \in \NN$ we define the  conical Radon transform   
$\Co_k f \colon \XXi \times \sph^1 \times \RR \to \RR$ with vertices on triple lines by 
\begin{equation}\label{eq:Ckf}
  \Co_{k}f(\uu,\bar{\boldsymbol\beta},\psi)
  \coloneqq 
  \begin{cases}
  \int_{\sph^2}
  \int^\infty_0 f(\uu+r\boldsymbol\alpha )r^{k}\delta(\boldsymbol\alpha\cdot \bbeta_\uu-s) \; \rmd r \; {\rm d}S(\aalpha) & \text{ if  $s \in [-1, 1]$}\\
  0 & \text{ otherwise}\,. 
\end{cases}
\end{equation}
Here  $\XXi $ and $\bbeta_\uu$ are   defined by \eqref{eq:xi}-\eqref{eq:beta},
$\delta$ is the one-dimensional delta-distribution and  ${\rm d}S$ is the standard area measure 
on the unit sphere $\sph^2$,
\begin{equation*}
{\rm d}S(\aalpha) = \delta\Big(1-\sqrt{\alpha_1^2+\alpha_2^2+\alpha_3^2}\Big)\rmd\alpha_1\rmd\alpha_2\rmd\alpha_3 \quad 
\text{ for  }
\aalpha = (\alpha_1,\alpha_2,\alpha_3)\in \mathbb{R}^3 \,.
\end{equation*} 
 \end{defi}

Assuming  that the density of photons decreases geometrically and proportional  to the distance 
from the source  to detectors, then the data  measured by a Compton camera are   given by the transform $\Co_1 f$.  When the density decreases at a different power of distance, we need  different values of $k$, see  \cite{moonip15}.

%

\section{Exact inversion formula}
\label{sec:inversion}

In this section we derive  an explicit inversion formula for the conical Radon  transform.
Along that way we introduce a weighted ray transform,  
show how the conical Radon transform can be reduced to the weighted ray transform and 
derive an explicit inversion formula for the weighted ray transform. 

\begin{defi}[Weighted ray transform]
Let $\XXi$ be defined by  \eqref{eq:xi} and set $\cone (\XXi)
\coloneqq \{\xx\in[0,1]^3 \colon x_1+x_2+x_3\leq1\}$.
We define the weighted ray transform $\Po_k f \colon \XXi \times \RR^3 \to \RR$ of  a continuous  compactly supported 
function $f \colon \RR^3 \to \RR$  by 
\begin{equation}\label{d-p}
\forall (\uu,\ww)  \in \XXi \times \RR^3 \colon \quad
\Po_kf(\uu,\ww) :=\intL\half f(\uu+r\ww) r^{k}  \rmd r \,.
\end{equation}
\end{defi}

It is easy to check that $\Po_k f$ is homogeneous of degree $- (k+1)$ in the variable of $\ww$, i.e., $\Po_kf(\uu,\ww)=|\ww|^{-(k+1)} \Po_kf(\uu,\ww / |\ww|)$ for $\ww\neq0$. 


\begin{lem}[Reduction of   the conical Radon transform to the ray transform]\label{lem:pkffromckf}
For $ f \in C^\infty(\RR^3)$ with compact support in $\cone (\XXi)$, we have
\begin{equation}\label{eq:wrtn2}
\forall (\uu,\ww)\in \XXi\times\RR^3 \colon \quad 
\Po_kf(\uu,\ww)=\frac{w_{\uu}}{4\pi|\ww|^{k+2}}\displaystyle\intL_{\sph^1}H_{s} \partial_s^{}(\Co_kf) \left(\uu,\bar\bbeta,\bar\bbeta\cdot\frac{\bar \ww_{\uu}}{|\ww|}\right) {\rm d}S(\bar\bbeta) \,.
\end{equation}
Here $H_s f (s) = \frac{1}{\pi} \int_{\RR} \frac{f(t)}{s-t} {\rm d}t$ is the Hilbert transform
and we have written  
 \begin{equation*}
\bar\ww_{\uu}
\coloneqq 
\begin{cases}
(w_1,w_3) \in \RR^2 &\text{ if } \uu \in   \XXi_1\\
(w_2,w_3) \in \RR^2 &\mbox{ if } \uu \in  \XXi_2\\
(w_1,w_3) \in \RR^2 &\mbox{ if } \uu \in   \XXi_3 
\end{cases}
\end{equation*}
and $w_{\uu}$ for the component of $\ww$ missing in $\bar\ww_{\uu}$.
\end{lem}

We omit the proof here since it can be similarly proved
as in that of Theorem 5 in \cite{moonsiims16} with some minor modification. 
For more details, we refer the readers to \cite{moonsiims16}.

Now we are ready to obtain the inversion formula for $\Co_kf$.

\begin{thm}[Inversion formula for the conical Radon transform]
\label{thm:inversionofck}
Define  the constant  $ c_k \coloneqq \frac{1}{2^2(2\pi)^3(k-1)!}$ and let 
$f\in C^\infty(\RR^3)$ have compact support in $\cone(\boldsymbol\Xi)$.
Then,
\begin{equation} \label{eq:main}
\begin{array}{ll}
f(\xx)\displaystyle=c_k \Delta_{\xx}\left\{\begin{array}{ll}\displaystyle
\intL_{\sph^2}\intL_{\mathbf n^\perp\cap \sph^2}\intL^{\alpha_1}_{-\infty}\intL_{\sph^1}
\displaystyle\frac{\alpha_2(\alpha_1-\omega)^{k-2}}{|(\omega,\alpha_2,\alpha_3)|^{k+2}}
(\partial_{u_1}^{k-1} H_{s} \partial_s^{}\Co_kf)\left(\Lambda(\xx,\mathbf n),\bar\bbeta,\bar\bbeta\cdot(\omega,\alpha_3)\right)\\
\qquad\qquad \left.\times{\rm d}\omega {\rm d}S(\bar\bbeta) {\rm d}S(\aalpha){\rm d}S(\mathbf n)\right.\hspace{0.1\textwidth} \text{ if } \Lambda(\xx,\mathbf n) \in  \XXi_1\\
 \displaystyle\intL_{\sph^2}\intL_{\mathbf n^\perp\cap \sph^2}\intL^{\alpha_2}_{-\infty}\intL_{\sph^1}
\displaystyle\frac{\alpha_1(\alpha_2-\omega)^{k-2}}{|(\omega,\alpha_2,\alpha_3)|^{k+2}}
(\partial_{u_2}^{k-1} H_{s} \partial_s^{}\Co_kf)\left(\Lambda(\xx,\mathbf n),\bar\bbeta,\bar\bbeta\cdot(\omega,\alpha_3)\right)\\
\qquad\qquad\left.\times{\rm d}\omega {\rm d}S(\bar\bbeta) {\rm d}S(\aalpha){\rm d}S(\mathbf n)\right.\hspace{0.1\textwidth} \text{ if } \Lambda(\xx,\mathbf n) \in \XXi_2\\
\displaystyle
\intL_{\sph^2}\intL_{\mathbf n^\perp\cap \sph^2}\intL^{\alpha_3}_{-\infty}\intL_{\sph^1}
\displaystyle\frac{\alpha_2(\alpha_3-\omega)^{k-2}}{|(\alpha_1,\alpha_2,\omega)|^{k+2}}
(\partial_{u_3}^{k-1} H_{s} \partial_s^{}\Co_kf)\left(\Lambda(\xx,\mathbf n),\bar\bbeta,\bar\bbeta\cdot(\alpha_1,\omega)\right)\\
\qquad\qquad\left.\times{\rm d}\omega {\rm d}S(\bar\bbeta) {\rm d}S(\aalpha){\rm d}S(\mathbf n)\right.\hspace{0.1\textwidth} \text{ if } \Lambda(\xx,\mathbf n) \in  \XXi_3 \,,\end{array}\right.
\end{array}
\end{equation}
where $\Lambda(\xx,\mathbf n)$ is any point in $\XXi\cap \{\yy\in\RR^3 \colon (\xx-\yy)\cdot\mathbf n=0\}$.
\end{thm}

\begin{proof}
Notice that by the chain rule, we have for $\uu\in\XXi_1$
$$
\begin{array}{ll}
(\partial_{w_1}^{k-1}\Po_1f)(\uu,\ww)&\displaystyle=\intL^\infty_0 r\partial^{k-1}_{w_1}f(\uu+r\ww){\rm d} r\\
&\displaystyle=\intL^\infty_0 \partial^{k-1}_{u_1}f(\uu+r\ww)r^k{\rm d} r=(\partial_{u_1}^{k-1} \Po_k f)(\uu,\ww)\,.
\end{array}
$$
Similarly, we have $(\partial_{w_2}^{k-1}\Po_1f)(\uu,\ww)=(\partial_{u_2}^{k-1} \Po_k f)(\uu,\ww)$ for $\uu\in\{0\}\times[0,1]\times\{0\}$ and
$(\partial_{w_3}^{k-1}\Po_1f)(\uu,\ww)=(\partial_{u_3}^{k-1} \Po_k f)(\uu,\ww)$ for $\uu\in\XXi_3$.
Together with Cauchy's formula for repeated integration we obtain
$$
\Po_1f(\uu,\ww)=\frac{1}{(k-1)!}\left\{\begin{array}{ll}\displaystyle \intL^{w_1}_{-\infty}(w_1-\omega)^{k-2}(\partial_{u_1}^{k-1} \Po_k f)(\uu,\omega,w_2,w_3){\rm d}\omega 
&\text{ if }  \uu \in  \XXi_1\\
\displaystyle\intL^{w_2}_{-\infty}(w_2-\omega)^{k-2}(\partial_{u_2}^{k-1} \Po_k f)(\uu,w_1,\omega,w_3){\rm d}\omega
&\text{ if } \uu \in \XXi_2\\
\displaystyle\intL^{w_3}_{-\infty}(w_3-\omega)^{k-2}(\partial_{u_3}^{k-1} \Po_k f)(\uu,w_1,w_2,\omega){\rm d}\omega
&\text{ if }  \uu \in  \XXi_3 \,.
\end{array}\right.
$$

Recall that $Rf$ be the 3-dimensional regular Radon transform, i.e., 
$$Rf(\mathbf n,s)
= \int_{\{ \xx\cdot \mathbf n = s\}} f(\xx) {\rm d}\xx
\qquad \text{for} \quad(\mathbf n,s)\in \sph^{2}\times\RR.$$ 
Then, using the polar coordinates, one can easily verify that 
\begin{equation}\label{eq:rffrompf}
\begin{array}{ll}
Rf(\mathbf n,\xx \cdot \mathbf n)
&\displaystyle=\intL_0^\infty \intL_{\mathbf n^\perp\cap \sph^2} f(\Lambda(\xx,\mathbf n)+ r \aalpha ) r {\rm d}S(\aalpha) {\rm d}r 
=\intL_{\mathbf n^\perp\cap \sph^2}\Po_1f\left(\Lambda(\xx,\mathbf n),\aalpha\right){\rm d}S(\aalpha)\\
&\displaystyle=\frac{1}{(k-1)!}\left\{\begin{array}{ll}\displaystyle\intL_{\mathbf n^\perp\cap \sph^2}\intL^{\alpha_1}_{-\infty}(\alpha_1-\omega)^{k-2}(\partial_{u_1}^{k-1} \Po_k f)(\Lambda(\xx,\mathbf n),\omega,\alpha_2,\alpha_3){\rm d}\omega{\rm d}S(\aalpha)\\\quad\hspace{0.1\textwidth} \text{ if } \Lambda(\xx,\mathbf n) \in  \XXi_1\\
\displaystyle\intL_{\mathbf n^\perp\cap \sph^2}\intL^{\alpha_2}_{-\infty}(\alpha_2-\omega)^{k-2}(\partial_{u_2}^{k-1} \Po_k f)(\Lambda(\xx,\mathbf n),\alpha_1,\omega,\alpha_3){\rm d}\omega{\rm d}S(\aalpha)\\\quad\hspace{0.1\textwidth} \text{ if } \Lambda(\xx,\mathbf n) \in\XXi_2\\
\displaystyle\intL_{\mathbf n^\perp\cap \sph^2}\intL^{\alpha_3}_{-\infty}(\alpha_3-\omega)^{k-2}(\partial_{u_3}^{k-1} \Po_k f)(\Lambda(\xx,\mathbf n),\alpha_1,\alpha_2,\omega){\rm d}\omega{\rm d}S(\aalpha)\\\quad\hspace{0.1\textwidth} \text{ if } \Lambda(\xx,\mathbf n) \in\XXi_3 \,.
\end{array}\right.
\end{array}
\end{equation}
We have the well-known inversion formula for $Rf$, i.e.,  for $\xx\in\RR^3$,
\begin{equation}\label{eq:inversionofradon3}
f(\xx)=\frac{1}{8\pi^2}\Delta_{\xx}\left(\intL_{\sph^2}Rf(\mathbf n , \xx\cdot \mathbf n ){\rm d}S(\mathbf n)\right).
\end{equation}
This implies that the function $f(\xx)$ is determined  by the integrals of  $Rf$
over the hyperplanes passing through the point $\xx$.  
Plugging \eqref{eq:rffrompf} into \eqref{eq:inversionofradon3}, we obtain the inversion formula for $f$ in terms of $\Po_1f$ as
\begin{equation}
\begin{split}
f(\xx) &=\frac{1}{8\pi^2(k-1)!}\Delta_{\xx}\left\{\begin{array}{ll}\displaystyle
\intL_{\sph^2}\intL_{\mathbf n^\perp\cap \sph^2}\intL^{\alpha_1}_{-\infty}(\alpha_1-\omega)^{k-2}(\partial_{u_1}^{k-1} \Po_k f)(\Lambda(\xx,\mathbf n),\omega,\alpha_2,\alpha_3){\rm d}\omega{\rm d}S(\aalpha){\rm d}S(\mathbf n)\\\quad\hspace{0.1\textwidth} \text{ if } \Lambda(\xx,\mathbf n) \in  \XXi_1\\
\displaystyle\intL_{\sph^2}\intL_{\mathbf n^\perp\cap \sph^2}\intL^{\alpha_2}_{-\infty}(\alpha_2-\omega)^{k-2}(\partial_{u_2}^{k-1} \Po_k f)(\Lambda(\xx,\mathbf n),\alpha_1,\omega,\alpha_3){\rm d}\omega{\rm d}S(\aalpha){\rm d}S(\mathbf n)\\\quad\hspace{0.1\textwidth} \text{ if } \Lambda(\xx,\mathbf n) \in\XXi_2\\
\displaystyle\intL_{\sph^2}\intL_{\mathbf n^\perp\cap \sph^2}\intL^{\alpha_3}_{-\infty}(\alpha_3-\omega)^{k-2}(\partial_{u_3}^{k-1} \Po_k f)(\Lambda(\xx,\mathbf n),\alpha_1,\alpha_2,\omega){\rm d}\omega{\rm d}S(\aalpha){\rm d}S(\mathbf n)\\\quad\hspace{0.1\textwidth} \text{ if } \Lambda(\xx,\mathbf n) \in\XXi_3 \,.
\end{array}\right.
\end{split}
\end{equation}
Using Lemma \ref{lem:pkffromckf}, we have our assertion.
\end{proof}

\begin{rmk}[Generalization  to different vertex sets] 
We point out that an  inversion formula similar to \eqref{eq:main} can also 
be derived for other arrangements triple line detectors.  In such a case,  one  
reconstructs a function $f\in C^\infty(\RR^3)$ with compact support in a certain set $K$ depending on 
$\XXi$  by deriving  generalizations of Lemma~\ref{lem:pkffromckf} and Theorem~\ref{thm:inversionofck}. For such results,  the following condition has to be satisfied: 
For every  $\xx \in K$, every plane passing through $\xx$ intersects the vertex set $\XXi$. 
\end{rmk}

\section{Conclusion}
 \label{sec:conclusion}

In this paper we derived an explicit inversion formula for inverting the conical Radon transform with vertices on triple lines. The considered geometry does not use formally over-determinated data and uses a bounded vertex set. As main auxiliary result we derived an inversion formula for a ray transform adjusted to the triple linear detector. While the used data was motivated by SPECT imaging  with one-dimensional Compton cameras our results are applicable to other settings as well. In future work we will investigate then  numerical implementation of the derived inversion approach and compare with other inversion methods.      
\section*{Acknowledgement}
The work of S. M. was supported by the National Research Foundation of Korea grant funded by the Korea government (MSIP) (2018R1D1A3B07041149). 

\bibliographystyle{plain}

\end{document}